\documentclass[preprint,12pt]{elsarticle}

\usepackage{amssymb}
\usepackage{amsthm}
\usepackage{amsmath}
\usepackage{graphics}
\usepackage{epsfig}
\usepackage{amssymb}
\usepackage{hyperref}
\usepackage{graphicx}
\usepackage{subfigure}
\usepackage{color}
\usepackage{tikz}
\usetikzlibrary{patterns}
\usepackage[ruled,vlined]{algorithm2e}
\usepackage{hyperref}
\usepackage{a4wide}
\usepackage{amsmath}
\usepackage{amsfonts}
\usepackage{enumerate}

\newtheorem*{theoremaux}{Theorem \theoremauxnum}
\gdef\theoremauxnum{1}

\newtheorem{lemma}{\bf Lemma}[section]
\newtheorem{note}{\bf Note}[section]

\newtheorem{theorem}{\bf Theorem}[section]

\newtheorem{corollary}[lemma]{\bf Corollary}
\newtheorem{definition}{\bf Definition}[section]

\journal{~}

\begin{document}
	
	\begin{frontmatter}
		
		
		\author{Sucharita Biswas}
		\ead{biswas.sucharita56@gmail.com}
		\address{Department of Mathematics,\\ Presidency University, Kolkata, India} 
		
		\author{Angsuman Das\corref{cor1}}
		\ead{angsuman.maths@presiuniv.ac.in}
		\address{Department of Mathematics,\\ Presidency University, Kolkata, India} 
		\cortext[cor1]{Corresponding author}
		
		\title{On a family of quasi-strongly regular Cayley graphs}

		\begin{abstract}
			In this paper, we construct a family of quasi-strongly regular Cayley graphs $\Gamma_H(G)$ which is defined on a finite group $G$ with respect to a subgroup $H$ of $G$. We also compute its full automorphism group	and characterize various transitivity properties	of it.		
		\end{abstract}
		
	\end{frontmatter}
	
	\section{Introduction} 
	Strongly regular graphs (SRG) is an important family of graphs generating attention of lot of researchers both due its rich theoretical implications as well as applications to various other fields. Consequently various generalizations of strong regularity, like {\it Deza graphs} \cite{Deza_1994},\cite{Deza_1999}, {\it generalized strongly regular graphs} \cite{gen_SRG}, {\it quasi strongly regular graphs} (QSRG) \cite{golightly-haynsworth-sarvate},\cite{QSRG_2021},\cite{QSRG_2023}, has emerged over time. In this paper, we construct a family of QSRG Cayley graphs by generalizing a well-known construction of SRG's from groups.
	
	On the other hand, characterizing the automorphisms and other related parameters involving symmetries of graphs is a relevant area of research. A few of the latest work in this direction are \cite{1/2-trans} \cite{Pappus} \cite{generalized-andrasfai} \cite{Cubic} \cite{Kneser} etc. In this paper, we characterize the full automorphism group and other transitivity properties of the proposed family of graphs.
	
	We start with few definitions which already exist in the literature.
	
	\begin{definition}
		A Cayley graph $Cay(G,S)$ of a group $G$ with respect to an inverse-symmetric subset $S$ of $G$ is a graph with $G$ as the set of vertices and two distinct vertices $u,v \in G$ are adjacent in $Cay(G,S)$ if $uv^{-1} \in S$. $S$ is said to be the connection set for $Cay(G,S)$.
	\end{definition}

	\begin{definition}\cite{goldberg}
		A strongly regular graph with parameters $(n,k,a,c)$, denoted by $SRG(n,$ $k,a,c)$, is a ` $k$' regular graph on ` $n$' vertices such that any two adjacent vertices have ` $a$' many common neighbours and any non-adjacent vertices have ` $c$' many common neighbours. 
	\end{definition}

	\begin{definition}\cite{golightly-haynsworth-sarvate} 
		A quasi-strongly regular graph with parameters $(n,k,a;c_1,c_2, \ldots ,c_p)$, denoted by $QSRG(n,k,a;c_1,c_2, \ldots ,c_p)$, is a ` $k$' regular graph on ` $n$' vertices such that any two adjacent vertices have ` $a$' many common neighbours and any two non-adjacent vertices have ` $c_i$' many common neighbours for some $1 \leq i \leq p$, where $c_i$'s are distinct non-negative integers. Let the set of parameters $\{ c_1,c_2, \ldots ,c_p\}$ is called $c-$set of $QSRG(n,k,a;c_1,c_2, \ldots ,c_p)$.  
	\end{definition}
	
	Now, we are in a position to introduce the proposed graph.	
	\begin{definition}
		Let $G$ be a group of order $n\geq 5$, $H$ be a subgroup of G and $S=\{(g,e),(e,g),(g,g): g \in G\setminus H \}$. Define $\Gamma_H(G)$ to be the Cayley graph of $G\times G$ with respect to the connection set $S$, i.e., two vertices $(g_1,g_2)$, $(g_3,g_4)$ are adjacent if and only if $(g_1g_3^{-1}, g_2g_4^{-1}) \in S$.
	\end{definition}
	Note that if $H=G$, then $\Gamma_H(G)$ is a null graph and if $H=\{e\}$, then $\Gamma_H(G)$ is a well-known family of strongly regular graphs with the parameters $(n^2,3n-3,n,6)$ as shown in \cite{nica-book}(page 27, Theorem 3.22). In other cases, i.e., if $\{e\} < H < G$, it can be shown that $\Gamma_H(G)$ is not strongly regular. The main contribution of this paper is to prove that $\Gamma_H(G)$ is a QSRG by evaluating its parameters and to compute its automorphism group for all non-trivial proper subgroups $H$ of $G$ and investigate the transitivity of such graphs.
	
	\section{$\Gamma_H(G)$ is QSRG}\label{qsrg-section}
	In this section, we prove that $\Gamma_H(G)$ is quasi strongly regular graph and evaluate all its parameters, depending upon various choices of $H$. As $\Gamma_H(G)$ is a Cayley graph, i.e., it is vertex-transitive, we deal with the vertex $(e,e)$ and emphasize on the adjacent and non adjacent vertices of $(e,e)$ to count their number of common neighbours.    
	
	The vertex $(e,e)$ is adjacent to vertices of the form $(g,e)$, $(e,g)$ and $(g,g)$, where $ g \in G\setminus H$. In this section, we denote this neighbours as Type $1$, $2$ and $3$ respectively.
	
	\begin{theorem}
		For any subgroup $H$ of $G$, any two adjacent vertices of $\Gamma_H(G)$ has $|G|-2|H|+2$ many common neighbours, i.e., $a=|G|-2|H|+2$.
	\end{theorem}
	
	\begin{proof}
		Let us consider the vertex $(e,e)$ in $\Gamma_H(G)$ and $(g,e)$ be a Type-$1$ neighbour of $(e,e)$. If $(g',e)$ is a common neighbour of $(e,e)$ and $(g,e)$, then $g,g' \notin H$ and $(g'g^{-1},e) \in S$, i.e., $g'g^{-1} \notin H$, i.e., $gH \neq g'H$. Thus $g' \in G \setminus \left( H \cup gH\right)$, i.e.,  we have $|G|-2|H|$ many choices for $g'$. 
		
		If $(e,g')$ is a common neighbour of $(e,e)$ and $(g,e)$, then $(g^{-1},g') \in S$, i.e., $g^{-1}=g'$. Thus we have only one choice for such a common neighbour, i.e., $(e,g^{-1})$.
		
		If $(g',g')$ is a common neighbour of $(e,e)$ and $(g,e)$, then $(g'g^{-1},g') \in S$, i.e., $g=g'$. In this case also, we have only one choice for such a common neighbour, i.e., $(g,g)$.
		
		Therefore $(e,e)$ and $(g,e)$ have $|G|-2|H|+2$ many common neigbours. 
		
		Similarly it can be proved that Type-$2$ and Type-$3$ neighbours of $(e,e)$  also have $|G|-2|H|+2$ many common neigbours with $(e,e)$. 
	\end{proof}
	
	Next we study the number of common neighbours of two non adjacent vertices, i.e., we will find the complete $c-$set for all subgroups $H$ of $G$ depending upon the nature of $H$. The complete result is given in Theorem \ref{mu-set}, which we prove next with the aid of some lemmas.
	
	\begin{theorem}\label{mu-set}
		Let $H$ be a non trivial proper subgroup of $G$ and $[G:H]=l$. 
		\begin{enumerate}
			\item If $l=2$ then $c-$set $=\{0,2,|G|-|H|\}$.
			\item If $l>2$ and $H$ is normal in $G$ then
			\begin{itemize}
				\item $|H|=2$ implies $c-$set $=\{2,6,|G|-|H|\}$.
				\item $|H|>2$ implies $c-$set $=\{0,2,6,|G|-|H|\}$.
			\end{itemize}
			\item If $l>2$ and $H$ is not normal in $G$  then
			\begin{itemize}
				\item $|H|=2$ implies $c-$set $=\{2,4,6,|G|-|H|\}$.
				\item $|H|>2$ implies $c-$set $=\{0,2,4,6,|G|-|H|\}$.
			\end{itemize}
		\end{enumerate}
	\end{theorem}

	\begin{lemma}\label{2_G-H}
		For any non trivial proper subgroup $H$ of $G$, we have $2,|G|-|H| \in c-$set. 
	\end{lemma}
	\begin{proof}
		Let $h \in H$. Clearly $(e,e)$ and $(e,h)$ are non-adjacent and have no common neighbours of Type $1$ and $3$. Let $(e,g)$ be a common neighbour of Type $2$ of $(e,e)$ and $(e,h)$. Hence from the adjacency condition we have $gh^{-1} \notin H$, i.e., $g \notin H$. Thus, we have $|G|-|H|$ choices for $g$, i.e., $(e,e)$ and $(e,h)$ have $|G|-|H|$ many common neighbours. Therefore  $|G|-|H| \in c-$set. 
		
		({\it Similarly, it can be shown that the number of common neighbours of $(e,e)$ and $(h,e)$ and the number of common neighbours of $(e,e)$ and $(h,h)$ are also $|G|-|H|$.})
		
		Now let $h \in H$ and $g \in G \setminus H$. Clearly $(e,e)$ and $(g,h)$ are non-adjacent and have no common neighbour of Type $1$.  Let $(e,g')$ be a Type-$2$ common neighbour of $(e,e)$ and $(g,h)$. Hence $(g^{-1}, g'h^{-1}) \in S$, i.e., $g'=g^{-1}h$. Therefore $(e,e)$ and $(g,h)$ have only one common neighbour of Type $2$. Similarly we can prove $(e,e)$ and $(g,h)$ have only one common neighbour of Type $3$. Therefore $(e,e)$ and $(g,h)$ have exactly $2$ common neighbours.  Hence $2 \in c-$set. 
		
		({\it Similarly, it can be shown that the number of common neighbours of $(e,e)$ and $(h,g)$ is also $2$})
	\end{proof}
	
	\begin{lemma}\label{g_1H=g_2H}
		Let $g_1,g_2 \in G\setminus H$ such that $g_1\neq g_2$ but $g_1H=g_2H$ and $Hg_1=Hg_2$. Then the number of common neighbours of $(e,e)$ and $(g_1,g_2)$ is exactly $2$.
	\end{lemma}
	
	\begin{proof}
		Clearly  $(e,e)$ and $(g_1,g_2)$ have no common neighbour of Type $3$.  Let $(g,e)$ be a common neighbour of Type $1$ of $(e,e)$ and $(g_1,g_2)$. Hence from the adjacency condition we have either $g=g_1$ or $g_2^{-1}g_1 \notin H$, i.e, $g_1H \neq g_2H$, which is a contradiction. Hence $(g_1,e)$ is the only common neighbour of Type $1$ of $(e,e)$ and $(g_1,g_2)$. Similarly it can be shown that $(e,g_2)$ is the only common neighbour of Type $2$ of $(e,e)$ and $(g_1,g_2)$.
	\end{proof}    
	
	\begin{lemma}\label{0_in_mu}
		If $|H|>2$, then $0 \in c-$set.
	\end{lemma}
	\begin{proof}
		Let $h_1, h_2 \in H \setminus \{e\}$, $h_1 \neq h_2$. Clearly $(e,e)$ and $(h_1,h_2)$ are non-adjacent and have no common neighbour of Type $3$. Let $(g,e)$ be a Type-$1$ common neighbour of $(e,e)$ and $(h_1,h_2)$. This implies $h_2=e$, a contradiction. Similarly if $(e,g)$ is a Type-$2$ common neighbour of $(e,e)$ and $(h_1,h_2)$, then we have $h_1=e$, which is a contradiction. Therefore $(e,e)$ and $(h_1,h_2)$ have no common neighbour, i.e., $0 \in c-$set.
	\end{proof}
	
	\begin{theorem}\label{H_index_2}
		Let $H$ be a non trivial proper subgroup of $G$ and $|G|>4$. If  $[G:H]=2$, then $c-$set $=\{0,2,|G|-|H|\}$.
	\end{theorem}
	\begin{proof}
		Vertices which are not adjacent to $(e,e)$ are of one of the $7$ forms $(h,e), (e,h), (h,h)$, $(h,h')$, $(g,h), (h,g), (g,g')$ where $h,h'\in H \setminus \{e\}$, $g,g' \in G \setminus H$ and $h\neq h'$, $g\neq g'$.
		
		From the proof of Lemma \ref{2_G-H}, it follows that the number of common neighbours of each of $(e,h),(h,e)$ and $(h,h)$ with $(e,e)$ is exactly $|G|-|H|$ and the number of common neighbours of each of $(g,h)$ and $(h,g)$ with $(e,e)$ is exactly $2$. Similarly, from Lemma \ref{0_in_mu}, it follows that the number of common neighbours of $(h,h')$ with $(e,e)$ is $0$. Hence $|G|-|H|, 0 \in c-$set.
		
		As $[G:H]=2$, hence $H$ is normal in $G$. Thus $gH=g'H$ and $Hg=Hg'$.  So by Lemma \ref{g_1H=g_2H},  it follows that the number of common neighbours of $(g,g')$ with $(e,e)$ is $2$. Hence $2 \in c-$set, this completes the proof. 
	\end{proof}

	\begin{lemma}\label{6_in_mu}
		Let $H$ be a non trivial proper subgroup of $G$. If  $[G:H]>2$ and $g_1H \neq g_2H$, $Hg_1 \neq Hg_2$ for some $g_1,g_2 \notin H$ then $6 \in c-$set.
	\end{lemma}
	\begin{proof}
		Consider the vertex $(g_1,g_2)$. If $(g,e)$ is a common neighbour of $(e,e)$ and $(g_1,g_2)$ then either $g=g_1$ or $g=g_2^{-1}g_1 \notin H$, i.e., $(g,e)=(g_1,e)$ or $(g_2^{-1}g_1,e)$. Thus the number of common neighbours of Type-$1$ of $(g_1,g_2)$ and $(e,e)$ is $2$.
		
		Similarly if $(e,g)$ is a common neighbour of $(e,e)$ and $(g_1,g_2)$, then either $g=g_2$ or $g=g_1^{-1}g_2 \notin H$. Thus the number of common neighbours of Type-$2$ of $(g_1,g_2)$ and $(e,e)$ is $2$.
		
		Lastly, if $(g,g)$ is a common neighbour of $(e,e)$ and $(g_1,g_2)$, then either $g=g_1$ and $g_1g_2^{-1} \notin H$ or $g=g_2$ and $g_2g_1^{-1} \notin H$. Thus the number of common neighbours of Type-$3$ of $(g_1,g_2)$ and $(e,e)$ is $2$.
		
		So $(e,e)$ and $(g_1,g_2)$ have $6$ common neighbours in total. Hence $6 \in c-$set.
	\end{proof}
	
	\begin{theorem}\label{H_normal}
		If $H$ is a normal subgroup of $G$ then 
		\begin{itemize}
			\item $|H|=2$ implies $c-$set $=\{2,6,|G|-|H|\}$.
			\item $|H|>2$ implies $c-$set $=\{0,2,6,|G|-|H|\}$.
		\end{itemize}
	\end{theorem}
	\begin{proof}
		Proceeding as in the proof of Theorem \ref{H_index_2}, vertices which are not adjacent to $(e,e)$ are of $8$ forms. One extra form arises due to the fact that in this case $(g,g')$, $gH\neq g'H$ may also occur.
		
		Rest of the proof follows as in Theorem \ref{H_index_2}, using Lemmas \ref{2_G-H}, \ref{g_1H=g_2H}, \ref{0_in_mu} and \ref{6_in_mu}.
	\end{proof}

	\begin{lemma}\label{non-normal-lemma}
		Let $G$ be a group and $H$ be a subgroup of $G$ which is not normal in $G$, then there exists \begin{enumerate}
			\item $g_1,g_2 \in G\setminus H$ such that $Hg_1=Hg_2$ but $g_1H \neq g_2H$.
			\item $g_1,g_2 \in G\setminus H$ such that $Hg_1\neq Hg_2$ but $g_1H = g_2H$.
		\end{enumerate}   
	\end{lemma}
	\begin{proof}
		As $H$ is not normal in $G$, there exist $a \in G \setminus H$ and $h \in H$ such that $a^{-1}ha \notin H$. Set $g_1=a$ and $g_2=ha$. Therefore $g_1g_2^{-1}=h^{-1} \in H$ which implies $Hg_1=Hg_2$ and $g_1^{-1}g_2=a^{-1}ha \notin H$ implies $g_1H \neq g_2H$.  The other part follows similarly.
	\end{proof}
	
	\begin{theorem}\label{H_not_normal}
		If $H$ is not a normal subgroup of $G$, then
		\begin{itemize}
			\item $|H|=2$ implies $c-$set $=\{2,4,6,|G|-|H|\}$.
			\item $|H|>2$ implies $c-$set $=\{0,2,4,6,|G|-|H|\}$.
		\end{itemize}
	\end{theorem}
	\begin{proof}
		As $H$ is not normal in $G$, by Lemma \ref{non-normal-lemma}(1), there exists $g_1, g_2 \in G \setminus H$ such that either $g_1H \neq g_2H$ and $Hg_1=Hg_2$. Consider the vertex $(g_1,g_2)$.
		Let $(g,e)$ be a common neighbour of $(e,e)$ and $(g_1,g_2)$. Hence either $g=g_1$ or $g=g_2^{-1}g_1 \notin H$. So $(e,e)$ and $(g_1,g_2)$ have $2$ common neighbours $(g_1,e)$ and $(g_2^{-1}g_1,e)$ of Type-$1$. Similarly we can prove  $(e,e)$ and $(g_1,g_2)$ have $2$ common neighbours $(e,g_2)$ and $(e,g_1^{-1}g_2)$ of Type-$2$. As $Hg_1=Hg_2$,  $(e,e)$ and $(g_1,g_2)$ have no common neighbour of Type-$3$. Hence they have total $4$ common neighbours, i.e., $4 \in c-$set.
		
		Since the cases of all other non-adjacent vertices of $(e,e)$ has been dealt with in the Lemmas \ref{2_G-H}, \ref{g_1H=g_2H}, \ref{0_in_mu} and \ref{6_in_mu}, the theorem follows.	
	\end{proof}
	
	\textbf{Proof of the Theorem \ref{mu-set}:} Combining the Theorems \ref{H_index_2}, \ref{H_normal} and \ref{H_not_normal}, we have the theorem. \qed
	
	\section{Automorphism Group of $\Gamma_H(G)$}\label{Automorphism-Section}
	In this section, we compute the full automorphism group of $\Gamma_H(G)$. To begin with, we fix some notations and prove some lemmas which will eventually be used to prove the main theorem of this section.
	
	Let $|H|=k$, $[G:H]=l+1$ and $\{ H, Ha_1, Ha_2, \ldots , Ha_l \}$ be the set of all right cosets of $H$ in $G$. Define 
	$$S_1=\{ (g,e) : g \in G \setminus H \}, S_2=\{ (e,g) : g \in G \setminus H \}, S_3=\{ (g,g) : g \in G \setminus H \} \mbox{ and}$$ 
	$$S^i_1=\{ (g,e) : g \in Ha_i \}, S^i_2=\{ (e,g) : g \in Ha_i \}, S^i_3=\{ (g,g) : g \in Ha_i \} \mbox{ for } i= 1,2, \ldots , l.$$  
	Hence  $$\bigcup_{i=1}^l S^i_1=S_1,~~ \bigcup_{i=1}^l S^i_2=S_2, ~~\bigcup_{i=1}^l S^i_3=S_3 \mbox{ and }\bigcup_{i=1}^3 S_i=S.$$ 
	
	Define $\langle S_i \rangle$ to be the subgraph of $\Gamma(G)$ induced by $S_i$ for $i=1,2,3$. Let $$\begin{array}{cc}
	H_1=\{ (h,e) : h \in H \setminus \{e \} \}; & C_1=\{(g,e):g \in G\setminus \{e\}\}\\
	H_2=\{ (e,h) : h \in H \setminus \{e \}  \}; & C_2=\{(e,g):g \in G\setminus \{e\} \}\\
	H_3=\{ (h,h) : h \in H \setminus \{e \}  \}; & C_3=\{(g,g):g \in G\setminus \{e\} \}
	\end{array}$$ 
	
	\begin{lemma}\label{complete_l_partite}
		Each of $\langle S_1 \rangle$,  $\langle S_2 \rangle$ and $\langle S_3 \rangle$ are complete $l$-partite graph with each partite set of size $k$, i.e., $$\langle S_1 \rangle \cong \langle S_2 \rangle \cong \langle S_3 \rangle \cong K_{ \underbrace{ k,k,\ldots , k}_{l~  times}}$$ 
	\end{lemma}
	\begin{proof}
		We will show that $\langle S_1 \rangle$ is a complete $l$-partite graph with partite sets $S_1^i$ and other cases can be treated similarly. Clearly $|S_1^i|=k$ for all $i=1,2, \ldots , l$. 
		
		Let $(g_1,e), (g_2,e) \in S^i_1$. If $(g_1,e) \sim (g_2,e)$, then $(g_1 g_2^{-1},e) \in S_1$, i.e., $g_1 g_2^{-1} \notin H$, i.e., $Hg_1 \neq Hg_2$ which is a contradiction as $g_1,g_2 \in Ha_i$. Hence $(g_1,e) \nsim (g_2,e)$. Thus each $S^i_1$ is an independent set.
		
		Now let $(g,e) \in S^i_1$ and $(g',e) \in S^j_1$, $i \neq j$. Hence $g \in Ha_i$, $g' \in Ha_j$, i.e., $Hg \neq Hg'$, i.e., $gg'^{-1} \notin H$, i.e., $(gg'^{-1},e) \in S_1$, i.e., $(g,e) \sim (g',e)$. This proves that it a complete $l$-partite graph. (See Figure \ref{picture}.)
	\end{proof}
	
	\begin{figure}[ht]
		\centering
		\begin{center}
			\includegraphics[scale=0.36]{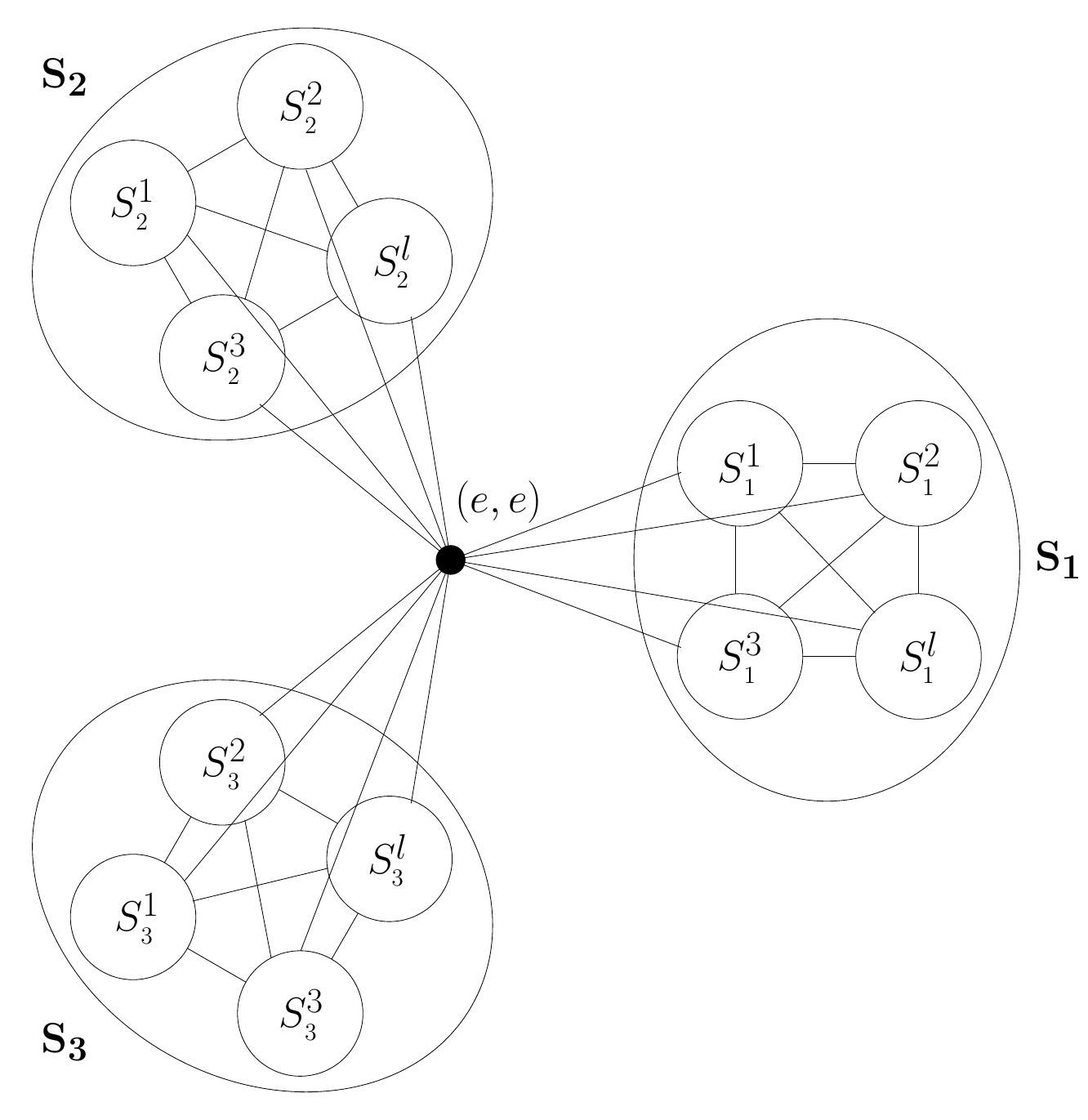}
			\caption{The neighbourhood of $(e,e)$ in $\Gamma_H(G)$}
			\label{picture}
		\end{center}
	\end{figure}

 Note that $S_i\subseteq C_i$ for $i=1,2,3$ and $\langle C_i \rangle$ is a complete $(l+1)$-partite graph where the partite sets are $H_i$ and $S^j_i$ for $j=1,2,\ldots,l$.
	
	Now we deal with two cases separately according as when $[G:H]=2$ and $[G:H]>2$.
	
	\begin{lemma}\label{complete_bi_partite}
		Let $[G:H]=2,|H|=k$ and $L$ be any complete bipartite subgraph of $\Gamma_H(G)$ with partite sets $X$ and $Y$ such that $|X|=k$  and $|Y|=k-1$. If $X \subset S$ and $Y \subset G \times G \setminus \{ (e,e) \cup S \}$, then $L=\langle C_1 \rangle$ or $\langle C_2 \rangle$ or $\langle C_3 \rangle$. 
	\end{lemma}
	\begin{proof}
		As $|G| \geq 5$, we have $k \geq 3$. Again, as $X\subset S$, any vertices in $X$ is either of Type $1$ or $2$ or $3$ (defined in Section \ref{qsrg-section}). Suppose $X$ contains a Type $1$ vertex, say $(a,e)$. If possible, let $(e,a') \in X$. Clearly $a' \neq a^{-1}$ as $(a,e)\not\sim (e,a')$. As $[G:H]=2$, we have $Ha =Ha'$. Let $(x,y)$ be a common neighbour of $(a,e)$ and $(e,a')$. 
		
		From the adjacency conditions we get $(x,y)=(a,a')$. Thus $(a,e)$ and $(e,a')$ has at most one common neighbour in $Y$.  On the other hand, as $L$ is a complete bipartite graph, any two vertices in $X$ must have exactly $k-1\geq 2$ common neighbours in $Y$, a contradiction. Hence $(e,a') \notin X$, i.e., $X$ does not contain any Type $2$ vertices. Similarly we can prove that $X$ does not contain any Type $3$ vertices. Thus $X$ consists of only Type $1$ vertices. As $|X|=k$, we have $X=S_1$, the set of all Type $1$ vertices. Now, as any vertex in $Y$ is adjacent to all vertices in $X$, we have $Y \subseteq H_1$. As $|Y|=k-1$, we get $Y=H_1$. Hence $L=\langle C_1 \rangle$. 
		
		Similarly, if we start with a Type $2$ or Type $3$ vertex in $X$, we get $L=\langle C_2 \rangle$ and $\langle C_3 \rangle$ respectively. Hence the theorem follows.
	\end{proof}
	
	\begin{lemma}\label{complete_l+1_partite}
		Let $[G:H]=l+1>2$ and $|H|=k$. Let $L=(X_1,X_2,\ldots, X_l,Y)$ be any complete $(l+1)$-partite subgraph of $\Gamma_H(G)$ such that $X_i \subseteq S$ with $|X_i|=k$ for $i=1,2,\ldots, l$ and $Y \subset G \times G \setminus \{ (e,e) \cup S \}$ with $|Y|=k-1$. Then $L=\langle C_1 \rangle$ or $\langle C_2 \rangle$ or $\langle C_3 \rangle$. 
	\end{lemma}
	\begin{proof}
		We first begin by proving the following claim:
		
		{\it Claim:} If $T \subseteq S$ such that $\langle T \rangle \cong K_{ \underbrace{ k,k,\ldots , k}_{l~  times}}$, then $T=S_1 $ or $S_2$ or $S_3$.
		
		{\it Proof of Claim:} We begin the proof by finding all the neighbours of the vertices $(a,e), (e,a)$ and $(a,a)$ in $S$ where $a \notin H$. 
		
		Let $(a,e) \sim (x,y)$. Hence $(xa^{-1},y) \in S$, i.e., either $xa^{-1} \notin H$ and $y=e$ or $xa^{-1}=e$ and $y \notin H$ or $xa^{-1}=y \notin H$. So, $(a,e)$ has $lk$ many neighbours of the form $(x,e)$ where $x \notin Ha$ and two other neighbours $(e,a^{-1})$, $(a,a)$ in $S$. 
		
		Similarly $(e,a)$ has $lk$ many neighbours of the form $(e,x)$ where $x \notin Ha$ and others are $(a^{-1},e)$, $(a,a)$ in $S$, and $(a,a)$ has $lk$ many neighbours of the form $(x,x)$ where $x \notin Ha$ and others are $(e,a)$, $(a,e)$ in $S$.
		
		Let $A_1, A_2, \ldots , A_l$ be the partite sets of $\langle T \rangle$. Let $(a_1,e) \in A_1$. If possible, let $(e,x) \in A_1$. Hence $(a_1,e) \nsim (e,x)$ and both have $(l-1)k$ many common neighbours in $S$. But, it can be shown  that $(a_1,e)$ and $(e,x)$ can have at most three common neighbours in $S$. If $(l-1)k>3$, we get a contradiction. Note that as $[G:H]=l+1>2$, the case $l=1$ does not arise. If $1\leq (l-1)k\leq 3$, we have finitely many cases where $|G|=3,4,5,6$ or $9$. Again, as $|G|\geq 5$, we are left with the case $|G|=5$ or $6$ or $9$. It can be checked manually that the lemma holds for groups of orders $5,6$ and $9$. Hence $(e,x) \notin A_1 $ for all $x \notin H$. 
		
		Similarly we can prove $(x,x) \notin A_1$ for all $x \notin H$. As $|A_1|=k$, the only possibility is $(x,e) \in A_1$ for all $x \in Ha_1$, i.e., $A_1=S_1^1$. Now by renaming suitably we can easily prove that $A_i=S_1^i$ for all $i =1,2,\ldots , l$. Hence $T=\langle S_1 \rangle$. 
		
		Similarly, if we start with the vertex $(e,a_1)$ and $(a_1,a_1)$ then we have   $T=\langle S_2 \rangle$ and $\langle S_3 \rangle$ respectively. Hence the claim follows.	
		
		Now consider $L=(X_1,X_2,\ldots, X_l,Y)$ to be any complete $(l+1)$-partite subgraph of $\Gamma_H(G)$ such that $X_i \subseteq S$ with $|X_i|=k$ for $i=1,2,\ldots, l$ and $Y \subset G \times G \setminus \{ (e,e) \cup S \}$ with $|Y|=k-1$. From the above claim, $L$ must contain $\langle S_1 \rangle$ or $\langle S_2 \rangle$ or $\langle S_3 \rangle$ as an induced subgraph. Hence if $L$ contains $\langle S_i \rangle$, then $Y$ must be equal to $H_i$, i.e., $L=\langle C_i \rangle$, for $i=1,2$ or $3$.	This completes the proof.
	\end{proof}

	Let $\mathcal{G}=Aut(\Gamma_H(G))$. As $\Gamma_H(G)$ is a Cayley graph on $G \times G$, $\mathcal{G}$ contains a copy of $G \times G$, namely $\{\psi_{(a,b)}: a,b \in G\}$, where  $\psi_{(a,b)}:G \times G \rightarrow G\times G$ is defined by $\psi_{(a,b)}(x,y)=(ax,by)$ with $a,b \in G$, is the left translation map. We start by noting two other automorphisms $\sigma$ and  $\alpha$ of $\Gamma(G)$ defined by,
	$$\sigma(x,y)=(y,x), ~\forall x,y \in G$$
	$$\alpha(x,y)=(y^{-1},y^{-1}x), ~\forall x,y \in G.$$ One can check that $\sigma$ is an automorphism of $\Gamma_H(G)$ with $\circ(\sigma)=2$ for all subgroups $H$ of $G$ and $\alpha$ is an automorphism of $\Gamma_H(G)$ with $\circ(\alpha)=3$ when $H$ is normal in $G$. Clearly $\sigma \alpha=\alpha^{-1}\sigma$, i.e., $\langle \alpha,\sigma \rangle \cong Sym(3)$, the symmetric group on three symbols. Let $Aut_H(G)=\{ f \in Aut(G) : f(H)=H \}$ be the set of all group automorphisms which fixes $H$ set-wise. Observe that every $f \in Aut_H(G)$  induces a graph automorphism $\tilde{f}:\Gamma_H(G) \rightarrow \Gamma_H(G)$ given by $$\tilde{f}(x,y)=(f(x),f(y)).$$
	
	In fact, $\iota: Aut_H(G)\rightarrow Aut(\Gamma_H(G))$ given by $\iota(f)=\tilde{f}$ is an injective group homomorphism. Thus $Aut_H(G)$ can be thought of as a subgroup of $Aut(\Gamma_H(G))$. Moreover, $\tilde{f}$ commutes with both $\alpha$ and $\sigma$ for all $f \in Aut_H(G)$ and $Aut_H(G)\cap \langle \alpha,\sigma\rangle$ is trivial. Thus  $Aut_H(G)\times Sym(3)$ is a subgroup of $\mathcal{G}$ when $H$ is normal in $G$ and for rest of the cases $Aut_H(G)\times \mathbb{Z}_2$  is a subgroup of $\mathcal{G}$. 
	
	Let $\mathcal{G}_0$ be the stabilizer of the vertex $(e,e)$ in $\mathcal{G}$. It is easy to observe that $Aut_H(G)\times Sym(3)$ is contained in $\mathcal{G}_0$  when $H$ is normal in $G$ and for rest of the cases, $Aut_H(G)\times \mathbb{Z}_2$ is a subgroup of $\mathcal{G}_0$.  Therefore any automorphism in $\mathcal{G}_0$ fixes the neighbourhood of $(e,e)$, $S$, as shown in Figure \ref{picture}. In fact, by Lemma \ref{complete_l_partite} and Claim of Lemma \ref{complete_l+1_partite}, any automorphism in $\mathcal{G}_0$ permutes $\langle S_1 \rangle$, $\langle S_2 \rangle$ and $\langle S_3 \rangle$. Now, by Lemma \ref{complete_bi_partite} and Lemma \ref{complete_l+1_partite}, it follows that any automorphism in $\mathcal{G}_0$ permutes $\langle C_1 \rangle$, $\langle C_2 \rangle$ and $\langle C_3 \rangle$. In the next lemma, we prove that any automorphism in $\mathcal{G}_0$ which fixes $\langle C_1 \rangle$, $\langle C_2 \rangle$ and $\langle C_3 \rangle$ belongs to $Aut_H(G)$.
	
	\begin{lemma}\label{Aut(G)}
		Let  $\varphi \in \mathcal{G}_0$ such that $\varphi(C_i)=C_i$ for $i=1,2,3$. Then $\varphi \in Aut_H(G)$.
	\end{lemma}
	\begin{proof}
		Let $x \notin H$, consider the path $P: (x,e) \sim (x,x) \sim (e,x)$. Let $\varphi((x,e))=(a_x,e)$, $\varphi((x,x))=(b_x,b_x)$ and $\varphi((e,x))=(e,c_x)$. Therefore $\varphi(P): (a_x,e) \sim (b_x,b_x) \sim (e,c_x)$. From the adjacency condition one can easily check that $a_x=b_x=c_x \notin H$.
		
		Now let $x \notin H$ and  $y \in H$. Consider the adjacencies 	
		$$\begin{array}{ccc}
		(x,x) \sim  (x,y)  \sim (e,y) & \varphi &(a_x,a_x) \sim  (x',y') \sim (e,c_y)\\
		\wr & \longrightarrow &  \wr\\
		(y,y) & &(b_y,b_y)
		\end{array}.$$
		Clearly $a_x \notin H$, $b_y, c_y \in H$ and $x' \neq y'$, $x',y' \neq e$. From the adjacency condition we have  $(x'a_x^{-1}, y'a_x^{-1})$, $(x'b_y^{-1}, y'b_y^{-1})$ and $(x',y'c_y^{-1}) \in S$. Hence we have  $x'=a_x$, $y'=c_y$ and $b_y=c_y$ i.e, $\varphi((x,y))=(a_x,b_y)$. Similarly we can show $\varphi((y,x))=(b_y,a_x)$.
		
		Now let $x, z \notin H$, $x \neq z$ and $y \in H$.  Consider the adjacencies $$\begin{array}{ccc}
		(x,y) & &(a_x,b_y)\\
		\wr & \varphi &  \wr\\
		(e,z) \sim  (x,z)  \sim (x,e) & \longrightarrow &(e,a_z) \sim  (x',z') \sim (a_x,e)\\
		\wr &  &  \wr\\
		(y,z) & &(b_y,a_z)
		\end{array}.$$
		Clearly $a_x, a_z \notin H$, $b_y \in H$ and $x' \neq z'$, $x',z' \neq e$. From the adjacency conditions we can have $\varphi((x,z))=(a_x,a_z)$.
		
		At last, let $x,y \in H$, $x \neq y$ and $z \notin H$. Consider the path $P: (x,z) \sim (x,y) \sim (z,y)$ and $\varphi(P): (b_x,a_z) \sim (x',y') \sim (a_z,b_y)$. Clearly $b_x, b_y \in H$, $a_z \notin H$ and $x' \neq y' $, $x',y' \in H \setminus \{e\}$.  From the adjacency conditions we can have $\varphi((x,y))=(b_x,b_y)$.
		
		For $x \notin H$ and $y \in H$  we have $(y,yx^{-1}) \sim (x,e)$. Let $\varphi((y,yx^{-1}))=(a_y,a_{xy^{-1}})$ and $\varphi((x,e))=(a_x,e)$. Therefore $(a_y,a_{yx^{-1}}) \sim (a_x,e)$, i.e., $(a_y a_x^{-1}, a_{yx^{-1}}) \in S$. As $x \neq y$, $a_x \neq a_y$ and $a_{yx^{-1}}\neq e$. Therefore from the adjacency condition we have $a_y a_x^{-1} = a_{yx^{-1}}$. Now take $y=e$, so $a_x^{-1}=a_{x^{-1}}$ and hence $a_y a_{x^{-1}} = a_{yx^{-1}}$. Finally, putting $x$ in place of $x^{-1}$, we get $a_ya_x=a_{yx}$.  As $\varphi((x,y))=(a_x,a_y)$, considering $\varphi$ as $(\varphi_1,\varphi_2)$, it is clear that $\varphi_1=\varphi_2=f$, where $f \in Aut_H(G)$. Thus $\varphi=\tilde{f}$, i.e., $\varphi \in Aut_H(G)$.
	\end{proof}
	
	\begin{corollary}\label{G_01}
		If $H$ is normal in $G$ then $\mathcal{G}_0\cong   Aut_H(G) \times Sym(3)$.
	\end{corollary}
	\begin{proof}
		Let $C=\{C_1,C_2,C_3\}$ and $Sym(C)$ be the symmetric group of the set $C$ which is isomorphic to $Sym(3)$. Define $\Psi: \mathcal{G}_0 \rightarrow Sym(C)$ such that $\Psi(f): C \rightarrow C$, $\Psi(f)(C_i)=f(C_i)$, $i=1,2,3$. One can easily check that $\Psi$ is a onto group homomorphism.  From the Lemma \ref{Aut(G)} we have  $ Ker(\Psi)=Aut_H(G) $, therefore  we have $ \mathcal{G}_0/Ker(\Psi) \cong   Sym(C)$, i.e., $\vert  \mathcal{G}_0 \vert=\vert Aut_H(G) \vert \vert Sym(C) \vert$. We already have $ Aut_H(G) \times Sym(3) \leq \mathcal{G}_0$, hence $\mathcal{G}_0 \cong  Aut_H(G) \times Sym(3)$.
	\end{proof}
	
	\begin{corollary}\label{G_02}
		If $H$ is not normal in $G$, then $\mathcal{G}_0\cong   Aut_H(G)\times \mathbb{Z}_2$.
	\end{corollary}
	\begin{proof}
		The proof follows similarly as the previous corollary.
	\end{proof}
	
	Now, we are in a position to conclude about the full automorphism groups of the graph $\Gamma_H(G)$.	
	\begin{theorem}
		The automorphism group $\mathcal{G}$ of $\Gamma_H(G)$ is the Zappa-Szep product of $G \times G$ and $\mathcal{G}_0$, i.e., $\mathcal{G}= \mathcal{L}\mathcal{G}_0$ with $|\mathcal{L}\cap \mathcal{G}_0|=1$, where $\mathcal{L}\cong G \times G$ denotes the left-translation subgroup of $\mathcal{G}$ and $\mathcal{G}_0$ is the stabilizer of the vertex $(e,e)$ in $\mathcal{G}$.
	\end{theorem}	
	\begin{proof}
		Clearly $\mathcal{L}\mathcal{G}_0 \subseteq \mathcal{G}$ and $|\mathcal{L}\cap \mathcal{G}_0|=1$. For the other direction, we invoke the orbit-stabilizer theorem for the action of $\mathcal{G}$ on $G\times G$. Since the action is transitive, we have $|\mathcal{G}|=|G \times G|\cdot |\mathcal{G}_0|$, which implies $\mathcal{G}= \mathcal{L}\mathcal{G}_0$.
	\end{proof}
	\begin{corollary}
		If $G$ is abelian, then  $\mathcal{L}$ is a normal subgroup of $\mathcal{G}$, i.e., $\mathcal{G}\cong \mathcal{L}\rtimes \mathcal{G}_0$.
	\end{corollary}
	\begin{proof}
		It is routine to check the normality of $\mathcal{L}$ in $\mathcal{G}$.
	\end{proof}
	
	\section{Transitivity of $\Gamma_H(G)$}	
	As $\Gamma_H(G)$ is a Cayley graph, it is vertex-transitive. However it may not be always edge-transitive. In this section, we investigate the necessary and sufficient conditions for edge(arc)-transitivity of $\Gamma_H(G)$. In fact, we prove the following theorem.
	
	\begin{theorem}
		Let  $G$  be a finite group and $H$ be a subgroup of $G$. Then the following are equivalent:
		\begin{enumerate}
			\item $\Gamma_H(G)$ is edge-transitive.
			\item $\Gamma_H(G)$ is arc-transitive.
			\item $H$ is normal in $G$ and $Aut_H(G)$ acts transitively on $G\setminus H$.
		\end{enumerate}
	\end{theorem}
	\begin{proof}
		Let $g_1,g_2 \in G\setminus H$. Consider the arcs $a_1:(e,e) \sim (g_1,e)$, $a_2:(e,e) \sim (g_2,e)$, $a_3:(e,e) \sim (e,g_2)$ and $a_4:(e,e) \sim (g_2,g_2)$. 
		\begin{itemize}
			\item $(1) \implies (3)$.	
			Let $\Gamma_H(G)$ is edge-transitive. At first we want to show $H$ is normal in $G$. If possible let $H$ is not normal in $G$, i.e., there exist $h \in H$, $g \in G \setminus H$ such that $ghg^{-1} \in G \setminus H$. Hence $\mathcal{G}_0\cong   Aut_H(G)\times \mathbb{Z}_2$.  Consider the arcs $a_5:(e,e) \sim (g,g)$, $a_6:(e,e) \sim (e,ghg^{-1})$. As $\Gamma_H(G)$ is edge-transitive, there exists $\varphi \in Aut(\Gamma_H(G))$ such that $\varphi(a_5)=a_6$.
			
			\textbf{Case 1:} Let $\varphi(\overrightarrow{a_5})=\overrightarrow{a_6}$, i.e., $\varphi((e,e))=(e,e)$ and $\varphi((g,g))=(e,ghg^{-1})$. Hence $\varphi \in \mathcal{G}_0$. Let $\varphi= \sigma^j \tilde{f}$, where $f \in Aut_H(G)$, $j\in \{0,1\}$. Hence $\varphi(g,g)=(f(g),f(g))\neq (e,ghg^{-1})$, which is a contradiction. 
			
			\textbf{Case 2:}  Let $\varphi(\overrightarrow{a_5})=\overleftarrow{a_6}$, i.e., $\varphi((e,e))=(e,ghg^{-1})$ and $\varphi((g,g))=(e,e)$. Let $\varphi=\psi_{(a,b)} \sigma^j \tilde{f}$, where $f \in Aut_H(G)$,  $j\in \{0,1\}$. Now  $\psi_{(a,b)} \sigma^j \tilde{f}((e,e))=(e,ghg^{-1})$ implies $a=e$, $b=ghg^{-1}$. Therefore $\psi_{(e,ghg^{-1})} \sigma^j \tilde{f}((g,g))=\psi_{(e,ghg^{-1})}(f(g),f(g))=(f(g), ghg^{-1}f(g)) \neq (e,e)$, which is a contradiction.  Hence $H$ is normal in $G$.
			
			Now we will prove $Aut_H(G)$ acts transitively on $G\setminus H$. As $\Gamma_H(G)$ is edge-transitive, hence there exists $\tau \in Aut(\Gamma_H(G))$ such that $\tau(a_1)=a_2$. 
			
			Let $\tau(\overrightarrow{a_1})=\overrightarrow{a_2}$, i.e., $\tau((e,e))=(e,e)$ and $\tau ((g_1,e))=(g_2,e)$.  Let $\tau=\alpha^i \sigma^j \tilde{f}$, where $f \in Aut_H(G)$, $i\in \{0,1,2\}$, $j\in \{0,1\}$. Hence $\alpha^i \sigma^j \tilde{f}((g_1,e))=(g_2,e)$ implies
			$f(g_1)=g_2$. As $g_1$, $g_2$ are arbitrary elements from $G\setminus H$, so $Aut_H(G)$ acts transitively on $G \setminus H$.
			
			Let  $\tau(\overrightarrow{a_1})=\overleftarrow{a_2}$, i.e., $\tau((e,e))=(g_2,e)$ and $\tau((g_1,e))=(e,e)$. Let $\tau=\psi_{(x,y)}\alpha^i \sigma^j \tilde{f}$, where $f \in Aut_H(G)$, $i \in \{0,1,2\}$,  $j\in \{0,1\}$. Now $\psi_{(x,y)}\alpha^i \sigma^j \tilde{f}((e,e))=(g_2,e)$ implies $x=g_2$, $y=e$. Therefore $\psi_{(g_2,e)}\alpha^i \sigma^j \tilde{f}((g_1,e))=(e,e)$ implies $f(g_1)=g_2^{-1}$.  Hence  $Aut_H(G)$ acts transitively on $G \setminus H$.   
			
			\item $(3) \implies (2)$.
			As $H$ is normal in $G$,  $\mathcal{G}_0\cong   Aut_H(G) \times Sym(3)$. As $Aut_H(G)$ acts transitively on $G \setminus H$, there exist $f_1, f_2 \in Aut_H(G)$ such that $f_1(g_1)=g_2$ and $f_2(g_1)=g_2^{-1}$.  Hence we can easily check that $\tilde{f_1}(\overrightarrow{a_1})=\overrightarrow{a_2} $,   $\sigma \tilde{f_1}(\overrightarrow{a_1})=\overrightarrow{a_3} $,  $\alpha^2 \tilde{f_2}(\overrightarrow{a_1})=\overrightarrow{a_4}$. Also note that  $\psi_{(g_2,e)}\tilde{f_2}(g_1,e)=(e,e)$ and $\psi_{(g_2,e)}\tilde{f_2}(e,e)=(g_2,e)$, hence $\psi_{(g_2,e)}\tilde{f_2}(\overrightarrow{a_1})=\overleftarrow{a_2}$. Similarly it can be checked that  $\psi_{(e,g_2)}\sigma \tilde{f_2}(\overrightarrow{a_1})=\overleftarrow{a_3}$ and $\psi_{(g_2,g_2)}\alpha^2 \tilde{f_1}(\overrightarrow{a_1})=\overleftarrow{a_4}$. Hence $\Gamma_H(G)$ is arc-transitive.
			
			\item $(2) \implies (1)$.
			This follows from the definition.
		\end{itemize}
	\end{proof}
	\begin{corollary}\label{elem-abelian}
		If $\Gamma_H(G)$ is edge-transitive, then $G/H$ is elementary abelian group.
	\end{corollary}
	\begin{proof}
		We start by  assuming that $H$ is normal in $G$ and $Aut_H(G)$ acts transitively on $G\setminus H$. Hence every element of $G \setminus H$ has same order. If possible let $\circ(x)=pq$ for all $x \in G \setminus H$, where $p,q$ are different primes. If $x^p, x^q \in H$, then $x \in H$, which is a contradiction. Therefore either $x^p$ or $x^q \in G \setminus H$. Which is a contradiction as $\circ(x) \neq \circ(x^p)$ and  $\circ(x) \neq \circ(x^q)$. Therefore $\circ(x) \neq pq$. 
		
		Let $\circ(x)=p^k$  for all $x \in G \setminus H$. So $\circ(x^p)=p^{k-1}$ imply $x^p \in H$, i.e., $(xH)^p=H$. Hence every element of $G/H$ has order $p$.
		Let $G'$ denote the commutator subgroup of $G$. 
		
		If possible let $G=G'$, i.e., $G$ is perfect, hence $G/H$ is perfect. As every element of $G/H$ has order $p$, $G/H$ is solvable, hence not perfect, which is a contradiction. Therefore $G' \subsetneq G$. As $G'$ is characteristic subgroup of $G$, $\varphi(G')=G'$ for all $\varphi \in Aut(G)$.
		
		Let $G' \nsubseteq H$, then there exists $g \in G' \setminus H$. If $G \setminus G' \nsubseteq H$, then there exists $g' \in G \setminus (G' \cup H)$. As $Aut_H(G)$ acts transitively on $G\setminus H$, there exists $\varphi \in Aut_H(G)$ such that $\varphi(g)=g' \notin G'$, which is a contradiction.  Hence $G \setminus G' \subseteq H$, i.e., $G=G' \cup H$. Hence either $G' \subseteq H$ or $H \subseteq G'$.
		If $H \subsetneq G'$, then there exist $g \in G'\setminus H$ and $g' \in G \setminus G'$. As $Aut_H(G)$ acts transitively on $G\setminus H$, there exists $\varphi \in Aut_H(G)$ such that $\varphi(g)=g' \notin G'$, which is a contradiction. Hence $G' \subseteq H$. Therefore $G/H$ is abelian. As every element of $G/H$ has order $p$, $G/H$ is elementary abelian group.
	\end{proof}
	
	\section{Conclusion and Open Issues}
	In this paper, we introduced a family of Cayley QSRGs and investigated its automorphism group and other transitivity properties. We conclude the paper by listing below some natural questions which arise from this work and can be of interest to researchers.
	\begin{enumerate}
		\item The converse of Corollary \ref{elem-abelian} is not true, in general. For example, one can check using SAGEMATH \cite{sage} or otherwise  that if  $G=\Bbb{Z}_4 \times \Bbb{Z}_2$ and $H \cong \Bbb{Z}_4$ be the unique cyclic subgroup of $G$ of order $4$, then though $G/H \cong \Bbb{Z}_2$, $\Gamma_H(G)$ is not edge-transitive. So it is natural to ask whether there exists some condition under which the converse holds?
		\item The family of QSRGs constructed in this paper are also Cayley graphs. Is it possible to classify all Cayley graphs which are QSRGs? 
	\end{enumerate}

	\section*{Acknowledgement}
	The first author is supported by the PhD fellowship of CSIR (File no. $08/155$ $(0086)/2020-EMR-I$), Govt. of India. The second author acknowledge the funding of DST-SERB-MATRICS Sanction no. $MTR/2022/000020$, Govt. of India.
	

\end{document}